\documentclass[11pt]{article}

\usepackage[a4paper,margin=1in]{geometry}
\usepackage[T1]{fontenc}
\usepackage[utf8]{inputenc}
\usepackage{lmodern}
\usepackage{microtype}
\usepackage{amsmath,amssymb,amsthm}
\usepackage{mathtools}
\usepackage{cite,xcolor}
\usepackage[nameinlink,noabbrev]{cleveref}
\usepackage{enumitem}
\newtheorem{theorem}{Theorem}[section]
\newtheorem{lemma}[theorem]{Lemma}
\newtheorem{proposition}[theorem]{Proposition}

\newtheorem{conjecture}[theorem]{Conjecture}
\theoremstyle{remark}

\newtheorem{question}[theorem]{Question}

\theoremstyle{plain}

\Crefname{theorem}{Theorem}{Theorems}
\Crefname{lemma}{Lemma}{Lemmas}
\Crefname{proposition}{Proposition}{Propositions}
\Crefname{corollary}{Corollary}{Corollaries}
\Crefname{conjecture}{Conjecture}{Conjectures}
\Crefname{remark}{Remark}{Remarks}
\Crefname{section}{Section}{Sections}

\newcommand{\rk}{\operatorname{rk}}
\newcommand{\B}{\mathcal B}
\newcommand{\G}{\mathcal G}
\newcommand{\supp}{\operatorname{supp}}

\usepackage{thm-restate}

\title{Neither simpliciality nor mutation connectivity:\\ conjectures of Las Vergnas and Cordovil--Las Vergnas fail}
\author{Qiyuan Gu\thanks{%
University of Chicago,
Chicago, IL, USA.
Email: \texttt{phoenix1203@uchicago.edu}.}
\and
Kolja Knauer\thanks{%
Departament de Matemàtiques i Informàtica,
Universitat de Barcelona, Barcelona, Spain;
Centre de Recerca Matemàtica (CRM),
Campus de Bellaterra, Edifici C,
08193 Bellaterra, Barcelona, Spain.
Email: \texttt{kolja.knauer@ub.edu}.}
}
\date{}
\begin{document}
\maketitle

\begin{abstract}
We construct a simple rank-$7$ oriented matroid on $24$ elements with no simplicial tope, disproving the Las Vergnas simplex conjecture from 1980. We then show that the mutation graph on uniform oriented matroids is disconnected for infinitely many values of rank $r$ and ground-set size $n$, disproving the Cordovil--Las Vergnas conjecture from 1988. 
\end{abstract}

\section{Introduction}

Oriented matroids were introduced by Bland--Las Vergnas and Folkman--Lawrence as an abstraction of real linear geometry~\cite{BL78,FL78}. They provide a broad setting in which to model, describe, and analyze combinatorial properties of geometric configurations, such as point and vector configurations, convex polytopes, directed graphs, linear programming, hyperplane arrangements, and great sphere arrangements. For introductions and complementary viewpoints we refer to the monographs \cite{BLSWZ99,Bok06,And25} and to the dynamic survey \cite{ZAK24}.

While oriented matroids preserve much of linear geometry, including duality and Farkas-type theorems, the nonrealizable world is substantially richer. For instance,
Pappus' theorem fails, face lattices of oriented matroids need not admit
polars, and the classical theorems of Weyl and Minkowski do not extend to
oriented matroids \cite{BLSWZ99,BilleraMunson84}. Moreover,
recognizing which oriented matroids are realizable is already
$\exists\mathbb R$-complete in rank~$3$ \cite{Mne88,Shor91}. For every fixed rank $r\geq 3$, asymptotically almost all oriented
matroids are nonrealizable \cite{BLSWZ99}.

There is however a strong positive result. A cornerstone distinguishing oriented matroid theory from ordinary matroid theory is the Topological Representation Theorem. It shows that, by relaxing arrangements of great spheres arising from real hyperplanes to arrangements of pseudospheres, every simple oriented matroid can be represented topologically \cite{FL78}.

One of the longest-standing open problems in the area, reflecting the remaining gap between topology and combinatorics in oriented matroid theory, is the Las Vergnas Simplex Conjecture. It says that in its topological representation:

\begin{conjecture}[Las Vergnas Simplex Conjecture~\cite{Las80}]\label{prob:LasVergnas}
 Every simple oriented matroid has a full-dimensional simplicial cell.
\end{conjecture}

The conjecture is supported by the fact that it holds for all realizable oriented matroids~\cite{Sha79} as well as for all rank-$3$ oriented matroids~\cite{fuk-93}. Bokowski and Rohlfs verified the uniform rank-$4$ conjecture computationally for $n<13$ \cite{BR01}, and Miyata proved it for uniform rank-$4$ matroid polytopes \cite{Miy20}.  
Mandel proved the conjecture for a substantially larger class containing Euclidean oriented matroids and posed the ``wishful thinking conjecture'' that every oriented matroid lies in that class \cite{Man82}.

In a negative direction, oriented matroids with a mutation-free element, i.e., one that does not participate in any simplicial cell, have been constructed on $20$ elements by Richter-Gebert \cite{Ric93}, on $17$ elements by Bokowski--Rohlfs~\cite{BR01}, and finally on $13$ elements by Hall~\cite{Hal04}. Hall's construction was a key ingredient in his proof that the cross-polytope is not extendably shellable in dimensions at least $12$. The construction later reappeared in learning theory, where Chalopin, Chepoi, Moran, and Warmuth used it to refute several earlier assertions about corner peelings and unlabeled sample compression \cite{CCMW22}. In \cite{KM23}, the existence of Hall's mutation-free element is used to disprove Mandel's ``wishful thinking conjecture'' \cite{Man82}. 

We use the parallel connections developed by Hochstättler and Nickel~\cite{HN11} on oriented matroids along mutation-free elements to disprove  
\Cref{prob:LasVergnas}. More precisely, for a simple oriented matroid $\mathcal M$, let $\delta(\mathcal M)$ denote the
minimum number of facets of a tope, and let
$\delta(r)=\sup\{\delta(\mathcal M):\mathcal M\text{ is simple of rank }r\}$. Then
\Cref{prob:LasVergnas} asserts that $\delta(r)=r$ for all $r$. However, we show:

\begin{restatable}{theorem}{mainSimplex}\label{thm:main-simplex}
For every integer $r\geq7$,
$$r+\left\lfloor\frac{r-1}{3}\right\rfloor-1
 \ \leq\ \delta(r)\ \leq\ 2r-3.$$
In particular, there is a simple rank-$7$ oriented matroid on $24$ elements with no simplicial tope.
\end{restatable}

This leaves the following question:

\begin{question}
What is the smallest $c$ such that $\delta(r)\leq cr$?
\end{question}

Note that the examples in the above theorem are far from being \emph{uniform}, i.e., their representing pseudosphere arrangements
are not in general position. The following is thus a very natural:

\begin{question}
    Are there uniform counterexamples to Las Vergnas' conjecture?
\end{question}

In the uniform setting, the Topological Representation Theorem is a perfect illustration of how topology and combinatorics can be intertwined.  If a uniform oriented matroid has a simplicial cell, then one can understand combinatorially what it means to transform the arrangement by ``pulling one pseudosphere bounding the simplicial cell over its opposite vertex''. This operation is called \emph{mutation} and leads to the notion of \emph{mutation graph} on the set of uniform oriented matroids of fixed rank on a fixed ground set. 

More precisely, following \cite{KM23}, for fixed $n,r$ we distinguish three graphs: $\overline{\G}^{n,r}$ has as vertices all labeled uniform rank-$r$ oriented matroids on one fixed $n$-element ground set, $\G^{n,r}$ has reorientation classes as vertices, and $\underline{\G}^{n,r}$ has isomorphism classes of reorientation classes as vertices. Edges are induced by mutations. There are natural quotient maps
\[
 \overline{\G}^{n,r}\longrightarrow \G^{n,r}\longrightarrow \underline{\G}^{n,r},
\]
which send edges to edges or collapse them, so connectivity propagates from left to right.
The following conjecture was recorded by {Roudneff and Sturmfels~\cite{RS88}, and Lawrence notes explicitly that they attribute it to Cordovil and Las Vergnas \cite{Law00}}:  

\begin{conjecture}[Cordovil--Las Vergnas~\cite{RS88}]\label{conj:Cordovil}
Any two uniform oriented matroids of fixed rank on a fixed ground set can be transformed into each other by mutations, i.e., the mutation graph $\overline{\G}^{n,r}$ is connected for all $n,r$.
\end{conjecture}

Ringel's Homotopy Theorem for pseudoline arrangements \cite{Rin56,Rin57} shows connectivity of $\G^{n,3}$
and \cite[Proposition~3.2]{KM23} extends this to $\overline{G}^{n,3}$. Roudneff and Sturmfels~\cite{RS88} proved that the subgraph of $\overline{G}^{n,r}$ induced by realizable oriented matroids is connected. Knauer and Marc verified connectivity of ${G}^{n,r}$ computationally for all $n\leq 9$~\cite{KM23}. 
Write $c(G)$ for the number of connected components of a graph $G$. Using weak-map preimages of the counterexample from \Cref{thm:main-simplex} we obtain:

\begin{restatable}{theorem}{mainMutation}\label{thm:main-mutation}
For every $r\geq 7$ and $n\geq r+17$, the mutation graphs
$\mathcal G^{n,r}$, $\underline{\mathcal G}^{n,r}$, and
$\overline{\mathcal G}^{n,r}$ are disconnected.
Moreover, for every fixed $r\geq 10$, as $n\to\infty$,

$$
2^{\Omega(n^{r-8})}
\leq c(\underline{\mathcal G}^{n,r})
\leq c(\mathcal G^{n,r})
\leq c(\overline{\mathcal G}^{n,r})
\leq 2^{O(n^{r-1})}.$$
\end{restatable}

\begin{question}
    What is the  asymptotic behavior of $c(\underline{\mathcal G}^{n,r}), c(\mathcal G^{n,r}), c(\overline{\mathcal G}^{n,r})$?
\end{question}

\section{Oriented matroid preliminaries}\label{sec:prelim}

We recall only the standard oriented-matroid terminology used below; see
\cite{BLSWZ99,And25} for background. Let $E$ be a finite set. For a sign
vector $X\in\{0,+,-\}^E$, write
\[
 X^0=\{e\in E:X(e)=0\},\qquad
 \supp(X)=E\setminus X^0,
\]
for its \emph{zero-set} and \emph{support}, respectively. Call $X$ \emph{nonnegative} if $X\in\{0,+\}^E$. 
For $A\subseteq E$, the \emph{restriction} of $X$ to $A$ is denoted by
$X|_A$. We also write $X\setminus A:=X|_{
E\setminus A}$. For $A\subseteq E$, the \emph{reorientation} ${}_{-A}X$ is obtained
from $X$ by reversing all nonzero signs on $A$. For sign vectors
$X,Y$, their \emph{composition} is
\[
 (X\circ Y)(e)=\begin{cases}X(e),&X(e)\neq0,\\Y(e),&X(e)=0.\end{cases}
\]

Their separator is $S(X,Y)=\{e\in E:X(e)=-Y(e)\neq0\}$. An \emph{oriented matroid} is a pair $\mathcal{M}=(E,\mathcal L)$, where $E$ is its \emph{ground set} and $\mathcal L\subseteq\{0,+,-\}^E$ is a set of \emph{covectors} satisfying

\begin{enumerate}[leftmargin=3em]
\item[(CV0)] $0\in\mathcal L$,
\item[(CV1)] $X\in\mathcal L$ implies $-X\in\mathcal L$,
\item[(CV2)] $X,Y\in\mathcal L$ implies $X\circ Y\in\mathcal L$,
\item[(CV3)] if $X,Y\in\mathcal L$ and $e\in S(X,Y)$, then there is $Z\in\mathcal L$ with $Z(e)=0$ and $Z(f)=(X\circ Y)(f)$ for every $f\notin S(X,Y)$.
\end{enumerate}
For $A\subseteq E$, we denote by
${}_{-A}\mathcal M$ the oriented matroid obtained by reorienting all its
covectors on $A$.
Under the componentwise order generated by $0<+$ and $0<-$,
together with an artificial maximum $\hat 1$, the covectors form the
\emph{big face lattice} $\mathcal F(\mathcal M)$. The height of $\mathcal F(\mathcal M)$ minus $1$ is the \emph{rank} $\rk(\mathcal{M})$ of $\mathcal{M}$. 
The atoms of $\mathcal F(\mathcal M)$ are the \emph{cocircuits} of $\mathcal{M}$.

Two sign vectors $X,Y$ are \emph{orthogonal}, denoted $X\bot Y$ if $\supp(X)\cap \supp(Y)=\emptyset$ or there are $e,f\in E$ such that $X(e)Y(e)=+$ and $X(f)Y(f)=-$. The \emph{circuits} of $\mathcal{M}$ are defined by \emph{orthogonality}, i.e., they are the support-minimal nonzero sign vectors satisfying $X\bot Y$ for every $Y\in\mathcal{L}$.
A \emph{coloop} is an $e\in E$ such that there is a cocircuit $X$ with $\supp(X)=\{e\}$. A \emph{loop} is an $e\in E$ such that there is a circuit $X$ with $\supp(X)=\{e\}$. We say that $\mathcal{M}$ is \emph{simple} if all its circuits have support of size at least $3$.

The coatoms of
$\mathcal F(\mathcal M)$ are the \emph{topes} of $\mathcal M$. A
\emph{facet} $X$ of a tope $T$ is a covector covered by $T$. An element
$g\in E$ \emph{supports} such a facet if $X(g)=0$. If
$\mathcal M$ is simple, the topes are precisely its covectors with empty zero-set and every facet has singleton zero-set. 
In a simple rank-$r$ oriented matroid every tope has at least $r$ facets,
and it is \emph{simplicial} if it has exactly $r$. Equivalently, the interval
$[0,T]\subseteq \mathcal F(\mathcal M)$ is a Boolean lattice \cite{BLSWZ99,KM23}. We call $g\in E$ \emph{mutation-free}
if no simplicial tope has a facet supported by $g$.

For the second part of the paper an equivalent axiomatization of oriented matroids will be more convenient. For some integer $r$, called the \emph{rank}, an oriented
matroid can be described by a \emph{chirotope}, i.e., a nonzero alternating map
\[
 \chi:E^r\longrightarrow\{0,+,-\}
\]
satisfying the \emph{Grassmann--Plücker} relations \cite[Chapter~3]{BLSWZ99}. On chirotopes the reorientation of a set $A\subseteq E$  multiplies $\chi(B)$ by $(-1)^{|B\cap A|}$. 
The chirotopes $\chi$ and $-\chi$ define the same oriented matroid; we refer to this ambiguity as the \emph{global chirotope sign}. An $r$-subset
$B\subseteq E$ is a \emph{basis} if $\chi(B)\neq0$, and $\mathcal M$ is
\emph{uniform} if every $r$-subset is a basis. In this language an element is a \emph{loop}
if it belongs to no basis and a \emph{coloop} if it belongs to every basis.  

For a uniform oriented
matroid, a \emph{mutation} is an $r$-subset $B$ for which reversing
$\chi(B)$, and only the values forced by alternation, again gives a
chirotope. Equivalently, $B$ is a mutation exactly when some simplicial
tope has its facets supported by the elements of $B$
\cite[Section~7.3]{BLSWZ99}. Reorientation and relabeling preserve mutations.

\section{Counterexample to the Las Vergnas simplex conjecture}\label{sec:twosum}

We need the oriented parallel connection developed by Hochst\"attler and Nickel~\cite{HN11}. Let $\mathcal M_i=(E_i,\mathcal L_i)$, $i=1,2$, have ranks $r_i$, let $E_1\cap E_2=\{g\}$, and assume that $g$ is neither a loop nor a coloop of either summand. Write $\mathcal C_i$ and $\B_i$ for the sets of signed circuits and bases of $\mathcal M_i$, respectively. By~\cite[Proposition~5]{HN11}, the \emph{oriented parallel connection} $\mathcal P=\mathcal M_1\oplus_P \mathcal M_2$ is the oriented matroid on $E_1\cup E_2$ whose covectors are precisely
\begin{equation}\label{eq:parallel-covectors}
 \mathcal L({\mathcal P})
 =\{X\in\{0,+,-\}^{E_1\cup E_2}:X|_{E_i}\in\mathcal L_i,\ i=1,2\}.
\end{equation}
For the circuit description, extend each $C_i\in\mathcal C_i$ by zero outside $E_i$. Then the signed circuits of $\mathcal P$ are given by~\cite[Section~3.1]{HN11}

\begin{equation}\label{eq:parallel-circuits}
  \mathcal C(\mathcal P)=\mathcal C_1\cup\mathcal C_2
  \cup\bigl\{(C_1\circ C_2)^{g\leftarrow 0}:C_i\in\mathcal C_i\ (i=1,2), g\in S(C_1,C_2)\bigr\},
 \end{equation}
where $(C_1\circ C_2)^{g\leftarrow 0}$ is the sign vector obtained from by setting the $g$-coordinate to $0$.
The bases of $\mathcal P$ are given by~\cite[Section~3.3]{HN11}
\begin{equation}\label{eq:parallel-bases}
 \begin{aligned}
  \B(\mathcal P)={}&
   \{B_1\cup B_2:B_i\in\B_i\ (i=1,2),\ g\in B_1\cap B_2\}\\
  &{}\cup\bigl\{(B_1\cup B_2)\setminus\{g\}:B_i\in\B_i\ (i=1,2),\
       g\in B_1\triangle B_2\bigr\},
 \end{aligned}
\end{equation}
where $\triangle$ denotes symmetric difference. In particular, since $g$ is a nonloop in each summand, 
\begin{equation}\label{eq:parallel-rank}
 \rk(\mathcal M_1\oplus_P\mathcal M_2)=r_1+r_2-1.
\end{equation}
As a last ingredient, for $A\subseteq E$ we use the standard \emph{deletion} formula
\begin{equation}\label{eq:deletion-covectors}
 \mathcal L(\mathcal M\setminus A)
 =\{X|_{E\setminus A}:X\in\mathcal L(\mathcal M)\}.
\end{equation}

The heart of the construction is the following, where iterated parallel connections are taken successively at the same element $g$.

\begin{lemma}\label{prop:connection-facets}
Let $\mathcal M_i=(E_i,\mathcal L_i)$, $i\in[k]$, be simple oriented matroids of ranks $r_i$, where $k\geq1$, the sets $E_i\setminus\{g\}$ are pairwise disjoint, and $g$ is neither a loop nor a coloop and is mutation-free in every $\mathcal M_i$. Let
\[
 \mathcal P=\mathcal M_1\oplus_P\cdots\oplus_P \mathcal M_k
\]
be their iterated oriented parallel connection at $g$. Then $\mathcal P\setminus g$ is simple of rank
\[
 \rk(\mathcal P\setminus g)=1+\sum_{i=1}^k(r_i-1),
\]
has $\sum_i(|E_i|-1)$ elements, and every tope of $\mathcal P\setminus g$ has at least $\sum_{i=1}^k r_i$ facets.
\end{lemma}

\begin{proof}
Iterating~\eqref{eq:parallel-rank} gives
$
\rk(\mathcal P)=1+\sum_{i=1}^k(r_i-1).
$
Moreover, by~\eqref{eq:parallel-bases}, choosing in every summand a basis containing $g$ gives a basis of
$\mathcal P$ containing $g$, while choosing in one summand a basis avoiding
$g$ and in all other summands bases containing $g$ gives a basis of
$\mathcal P$ avoiding $g$. Thus $g$ is neither a loop nor a coloop of
$\mathcal P$. Hence, deleting $g$ preserves the rank and leaves $\sum_i(|E_i|-1)$ elements.

Iterating~\eqref{eq:parallel-circuits} shows that $\mathcal P$ is simple. Indeed, circuits inherited from a summand have size at least $3$, while a circuit created by joining two circuits through $g$ contains, after deleting $g$, at least $2$ elements from each side. Since deletion only removes circuits, $\mathcal P\setminus g$ is simple, too.

To see the lower bound on the number of facets, let $T$ be any tope of $\mathcal P\setminus g$. By~\eqref{eq:deletion-covectors}, choose a covector $X$ of $\mathcal P$ with $X\setminus g=T$. Iterating~\eqref{eq:parallel-covectors} implies that $X_i:=X|_{E_i}$ is a covector of $\mathcal M_i$ for every $i$. If $X(g)\neq0$, then $X_i$ is a tope of $\mathcal M_i$ and we set $\sigma=X(g)$. If $X(g)=0$, then $X_i^0=\{g\}$, and we choose $\sigma\in\{+,-\}$ arbitrarily. Composing $X_i$ with a tope of $\mathcal M_i$ whose sign at $g$ is $\sigma$ fills this zero. Hence in either case we may choose topes $T_i$ of $\mathcal M_i$ such that $
T_i\setminus g=T|_{E_i\setminus\{g\}}$ and $T_i(g)=\sigma$.

We show that $T_i$ has at least $r_i$ facets not supported by $g$. Indeed, if no facet of $T_i$ is supported by $g$, all of its at least $r_i$ facets are available. If exactly one facet of $T_i$ is supported by $g$, then since $g$ is mutation-free $T_i$ is not simplicial, so $T_i$ has at least $r_i+1$ facets and again at least $r_i$ facets not supported by $g$. Since $\mathcal M_i$ is simple, every facet of $T_i$ has a singleton zero-set, and 
not more than one facet of $T_i$ is supported by $g$.

Now, for any facet $F_i<T_i$ not supported by $g$, define a sign vector $Y_i$ by keeping $F_i$ on $E_i$ and keeping $T_j$ on every other summand. Since the sign at $g$ is $\sigma$ for all of the $T_j$ and $F_i$, iterating~\eqref{eq:parallel-covectors} gives that $Y_i$ is a covector of $\mathcal P$. After deleting $g$, $Y_i$ agrees with $T$ everywhere except at the unique zero of $F_i$. Thus $Y_i\setminus g$ is a facet of $T$. Facets obtained from different summands are distinct. Thus $T$ has at least $\sum_i r_i$ facets.
\end{proof}

\mainSimplex*

\begin{proof}
By the average-subface theorem of Fukuda--Tamura--Tokuyama~\cite{FTT93}, the average number of facets of a tope in a rank-$r$ oriented matroid is strictly less than $2r-2$ for $r\ge3$.
Thus, some tope has at most $2r-3$ facets, proving the upper bound. 
For the lower bound let $r\geq7$, set $k=\left\lfloor\frac{r-1}{3}\right\rfloor$,\ with $t=r-(3k+1)\in\{0,1,2\}$, and
apply \Cref{prop:connection-facets} to $k$ copies of Hall's uniform
rank-$4$ oriented matroid on $13$ elements with a mutation-free element $g$~\cite{Hal04}. The deletion of $g$ is a simple rank-$(3k+1)$
oriented matroid on $12k$ elements with at least $4k$ facets at every
tope. Adding $t$ coloops preserves simplicity, raises the rank by $t$,
and adds one facet to every tope for each added coloop. Hence the resulting
rank-$r$ oriented matroid has at least
\[
 4k+t=r+\left\lfloor\frac{r-1}{3}\right\rfloor-1
\]
facets at every tope. For $r=7$ this construction uses two Hall summands and no coloops, so it
has $24$ elements and every tope has at least $8$ facets. In particular,
it has no simplicial tope.
\end{proof}

\section{Counterexamples to the Cordovil--Las Vergnas conjecture}\label{sec:weakmaps}
In this section we will make use of further standard oriented matroid notions that can also be found in~\cite{And25,BLSWZ99}. If
$B=(b_1,\ldots,b_r)$ is an ordering of a basis and $e\notin B$, the \emph{fundamental circuit}
$C(B,e)$ is the signed circuit with support equal to the unique circuit contained
in $B\cup\{e\}$, normalized by $C(B,e)(e)=+$. For $f\in B$, the
\emph{fundamental cocircuit} $D(B,f)$ is the signed cocircuit with support equal
to the unique cocircuit contained in $(E\setminus B)\cup\{f\}$, normalized by
$D(B,f)(f)=+$. Their supports intersect
in a subset of $\{f,e\}$, so orthogonality gives
\begin{equation}\label{eq:pm}
 C(B,e)(f)=-D(B,f)(e).
\end{equation}

Moreover,
\begin{equation}\label{eq:fundamental-cocircuit}
 D(B,f)(e)=\chi(b_1,\ldots,e,\ldots,b_r)\,\chi(b_1,\ldots,b_r),
\end{equation}
with $e$ in the position of $f$ \cite[Section~3.5]{BLSWZ99}.

For oriented matroids $\mathcal U,\mathcal M$ of the same rank on the same
ground set, we write $\mathcal U\rightsquigarrow\mathcal M$ for a
rank-preserving \emph{weak map} if, after choosing their global chirotope
signs compatibly, i.e., possibly replacing one chirotope by its negative, we have $ \chi_{\mathcal M}(B)\in\{0,\chi_{\mathcal U}(B)\}$ 
for every ordered $r$-subset $B$ \cite[Proposition~7.7.5]{BLSWZ99}. If $\mathcal U\rightsquigarrow\mathcal M$ is a rank-preserving weak map and
$B$ is a basis of both oriented matroids, then the fundamental-cocircuit
formula~\eqref{eq:fundamental-cocircuit} shows that the entries of the fundamental cocircuits
at $B$ can only become zero.
The crucial lemma of this section is:

\begin{lemma}\label{lem:weak-map-mutation}
Let $\mathcal M$ be a simple oriented matroid and let
$\mathcal U\rightsquigarrow\mathcal M$, where $\mathcal U$ is uniform.
If a basis $B$ of $\mathcal M$ is a mutation of $\mathcal U$, then
$\mathcal M$ has a simplicial tope whose facets are supported precisely by
the elements of $B$.
\end{lemma}
\begin{proof}
Reorient $\mathcal U$ and $\mathcal M$ simultaneously, which preserves the
weak map, so that a simplicial tope of $\mathcal U$ corresponding to the
mutation $B$ is positive. The fundamental cocircuits
$D_{\mathcal U}(B,b)$, $b\in B$ are the cocircuits corresponding to the atoms of the Boolean interval below that tope. Thus, they are nonnegative. Since the corresponding
cocircuits of $\mathcal M$ are obtained by replacing some entries by zero,
they too are nonnegative.

Since $\mathcal M$ has no loops, for every $e\notin B$ the fundamental
circuit $C_{\mathcal M}(B,e)$ contains some $b\in B$. Hence
$D_{\mathcal M}(B,b)(e)\neq0$ for some $b\in B$, and therefore
$D_{\mathcal M}(B,b)(e)=+$, as all these cocircuits are nonnegative.
Moreover, $D_{\mathcal M}(B,b)(b)=+$ for every $b\in B$. Thus the
composition of the $D_{\mathcal M}(B,b)$, $b\in B$, is the positive tope
$T$ of $\mathcal M$.

Let $F$ be a facet of $T$. Since $\mathcal M$ is simple,
$F^0=\{e\}$ for some $e\in E$. Suppose that $e\notin B$.
Then $F(b)=+$ for every $b\in B$. The fundamental circuit
$C_{\mathcal M}(B,e)$ satisfies $C_{\mathcal M}(B,e)(e)=+$,
and by~\eqref{eq:pm}
all of its nonzero entries on $B$ are negative. Since
$F(e)=0$, every nonzero product of corresponding entries of $F$ and $C_{\mathcal M}(B,e)$ is negative, contradicting orthogonality.

Thus every facet of $T$ is supported by an element of $B$. Since $\mathcal M$ is simple, a facet below the positive tope $T$ is determined by its singleton zero-set, so $T$ has at most $|B|=r$ facets. Every tope of a simple rank-$r$ oriented matroid has at least $r$ facets. Hence $T$ has exactly $r$ facets and is simplicial, with its facets supported precisely by the elements of $B$.
\end{proof}

Recall that, for fixed $n,r$, $\overline{\G}^{n,r}$ denotes the mutation graph on
labeled uniform rank-$r$ oriented matroids on a fixed $n$-element ground set,
$\G^{n,r}$ is its quotient by reorientation, and
$\underline{\G}^{n,r}$ is its quotient by relabeling and reorientation. Adjacency is induced by
mutations. 
\begin{proposition}\label{thm:fiber}
Let $\mathcal M$ be a simple oriented matroid of rank $r$ on $n$ elements with no
simplicial tope. Then the fiber

$$
 \mathcal W(\mathcal M)=\{[\mathcal U]\in V(\G^{n,r}):
       \text{some representative }\mathcal U\text{ satisfies }
       \mathcal U\rightsquigarrow\mathcal M\}
$$

is a nonempty proper union of connected components of $\G^{n,r}$. Moreover,
fibers over reorientation-inequivalent orientations of the same underlying
matroid are disjoint.
\end{proposition}

\begin{proof}

If
$\mathcal U\rightsquigarrow\mathcal M$ has mutation  $B$, then by \Cref{lem:weak-map-mutation} since $\mathcal M$ has no simplicial tope, $B$ is not a
basis of $\mathcal M$, so $\chi_{\mathcal M}(B)=0$. The mutation therefore
preserves the weak map to $\mathcal M$. Mutations commute with
reorientations; so if an edge of $\G^{n,r}$ is represented by a mutation
of another representative of $[\mathcal U]$, reorienting both endpoints
back to $\mathcal U$ shows that the adjacent class again has a
representative weakly mapping to $\mathcal M$. Hence
$\mathcal W(\mathcal M)$ is a union of components. Further, the fiber is
nonempty because every oriented matroid admits a uniform weak-map preimage of the same rank and groundset\cite[Corollary~7.7.9]{BLSWZ99}.

To see that the fiber is proper, choose a basis $B$ of $\mathcal M$ and a
realizable uniform rank-$r$ oriented matroid $\mathcal R$ on $E$. By
Shannon's theorem, $\mathcal R$ has a simplicial tope and hence a mutation
\cite{Sha79}. Relabel $\mathcal R$ so that this mutation is $B$. Since
reorientation preserves mutations, no representative of $[\mathcal R]$
weakly maps to $\mathcal M$, by \Cref{lem:weak-map-mutation}. Thus
$[\mathcal R]\notin\mathcal W(\mathcal M)$.

Finally suppose $\mathcal U\rightsquigarrow\mathcal M$ and
${}_{-A}\mathcal U\rightsquigarrow\mathcal M'$ for a reorientation on $A\subseteq E$, and
$\mathcal M'$ has the same underlying matroid as $\mathcal M$. Then
$\mathcal U\rightsquigarrow{}_{-A}\mathcal M'$ as well. Both
$\mathcal M$ and ${}_{-A}\mathcal M'$ vanish exactly on the non-bases of
the common underlying matroid and take the value $\chi_{\mathcal U}(B)$,
up to a global sign, on every basis $B$. Hence
$\mathcal M={}_{-A}\mathcal M'$ up to the global chirotope sign, and
$\mathcal M,\mathcal M'$ are reorientation-equivalent.
\end{proof}


\begin{lemma}\label{lem:mutation-monotonicity}
Let $\mathcal H^{n,r}$ be any of
$\overline{\G}^{n,r}$, $\G^{n,r}$, and $\underline{\G}^{n,r}$.
If $\mathcal H^{n,r}$ is disconnected and $r'\ge r$ and
$n'-r'\ge n-r$, then $\mathcal H^{n',r'}$ is disconnected.
\end{lemma}

\begin{proof}
A map $f\colon V(H)\to V(K)$ is a \emph{weak graph homomorphism} if for every
edge $xy$ of $H$ either $f(x)=f(y)$ or $f(x)f(y)$ is an edge of $K$.
For a uniform rank-$r$ oriented matroid $\mathcal U$ on $E\cup\{e\}$,
where $|E|=n\ge r$, deletion is given by
$\chi_{\mathcal U\setminus e}(b_1,\ldots,b_r)=\chi_{\mathcal U}(b_1,\ldots,b_r)$
for $b_1,\ldots,b_r\in E$. For a uniform rank-$(r+1)$ oriented matroid
$\mathcal V$ on the same set, contraction is given, up to global sign, by
$\chi_{\mathcal V/e}(b_1,\ldots,b_r)=\chi_{\mathcal V}(e,b_1,\ldots,b_r)$
\cite[Section~3.5]{BLSWZ99}.
Consequently, a mutation at $B$ leaves the deletion unchanged if $e\in B$,
and otherwise changes only its value at $B$. It leaves the contraction
unchanged if $e\notin B$, and otherwise changes only its value at
$B\setminus\{e\}$, in each case together with the values forced by alternation.
Deletion and contraction of a fixed element give surjective weak graph
homomorphisms

$$
 \overline{\G}^{n+1,r}\longrightarrow\overline{\G}^{n,r},
 \qquad
 \overline{\G}^{n+1,r+1}\longrightarrow\overline{\G}^{n,r}.
$$

Surjectivity follows from the fact that every uniform oriented matroid
$\mathcal M$ admits a uniform single-element extension
$\widetilde{\mathcal M}$ with $\widetilde{\mathcal M}\setminus e=\mathcal M$
and a uniform single-element lifting $\widehat{\mathcal M}$ with
$\widehat{\mathcal M}/e=\mathcal M$
\cite[Sections~7.1--7.2]{BLSWZ99}. A surjective weak graph homomorphism
maps each component into a component and hits every component, so the
preimage has at least as many components as the image. Both maps descend to
reorientation classes, since deletion and contraction commute with
reorientation up to reorientation and global chirotope sign. Hence
disconnectedness propagates for $\overline{\G}$ and $\G$ under the stated
conditions. Connectivity of $\G$ and $\underline{\G}$ is equivalent
\cite[Observation~3.1 and Proposition~3.3]{KM23}, so the same holds for
$\underline{\G}$. Iterating proves the claim.
\end{proof}

\mainMutation*
\begin{proof}
By \Cref{thm:main-simplex}, there is a simple rank-$7$ oriented
matroid $\mathcal M$ on $24$ elements with no simplicial tope.
\Cref{thm:fiber} implies that
${\mathcal G}^{24,7}$ is disconnected, so $\overline{{\mathcal G}}^{24,7}$ is disconnected and by
\cite[Proposition~3.3]{KM23} as well $\underline{{\mathcal G}}^{24,7}$. \Cref{lem:mutation-monotonicity}
therefore gives disconnection whenever $r\geq 7$ and
$n-r\geq 17$.

For the quantitative statement, use the above $\mathcal M$
and let $q=|E(\mathcal M)|$. Let $s=r-7\geq 3$ and $m=n-q$.
For every uniform rank-$s$ oriented matroid $\mathcal U$ on a fixed
$m$-element set disjoint from $E(\mathcal M)$, the direct sum
$\mathcal M\oplus\mathcal U$ is a simple rank-$r$ oriented matroid on
$n$ elements with no simplicial tope. Indeed, if
$T=T_{\mathcal M}\oplus T_{\mathcal U}$ is a tope, then $
[0,T]=[0,T_{\mathcal M}]\times[0,T_{\mathcal U}]$, 
which is not Boolean because $[0,T_{\mathcal M}]$ is not Boolean.

All these direct sums have the same underlying matroid, and restriction to the ground set of the second summand recovers $\mathcal U$.  
Distinct reorientation classes of $\mathcal U$ yield distinct
reorientation classes of $\mathcal M\oplus\mathcal U$.
The moreover-part of \Cref{thm:fiber} therefore gives pairwise disjoint
nonempty unions of components of ${\mathcal G}^{n,r}$, one for
each reorientation class of $\mathcal U$.

For fixed $s\geq 3$, the number of labeled uniform rank-$s$ oriented
matroids on $m$ elements is
$2^{\Theta(m^{s-1})}$~\cite[Corollary~7.4.3]{BLSWZ99}.
Since a reorientation class contains at most $2^m$ oriented matroids,
there are still $2^{\Theta(m^{s-1})}$ such classes. Hence $
c({\mathcal G}^{n,r})
   \geq 2^{\Omega(n^{r-8})}$. 
Passing to isomorphism classes identifies at most $n!$ components, so

$$
c(\underline{\mathcal G}^{n,r})
   \geq \frac{c({\mathcal G}^{n,r})}{n!}
   =2^{\Omega(n^{r-8})},
$$

because $\log_2(n!)=O(n\log n)=o(n^{r-8})$ for $r\geq10$. Finally, all three numbers of components are at most the
number of labeled uniform rank-$r$ oriented matroids on $n$ elements,
which is $2^{O(n^{r-1})}$ by the same counting result.
\end{proof}

\subsection*{Use of AI-assisted tools}
During exploratory work, the authors used OpenAI's ChatGPT to generate and test possible proof strategies, to assist with literature searches, and to improve exposition. In particular, it contributed to the development and checking of the parallel-connection construction in \Cref{sec:twosum} and of the weak-map fiber argument in \Cref{sec:weakmaps}. All mathematical statements, proofs, and citations in the final manuscript were independently checked by the authors, who take full responsibility for their correctness.

\subsection*{Acknowledgments}
KK was supported by the grant PID2022-137283NB-C22 funded by MICIU/AEI/10.13039/\allowbreak 501100011033 and by ERDF/EU and through the Severo Ochoa and Mar\'ia de Maeztu Program for Centers and Units of Excellence in R\&D (CEX2020-001084-M).

\begingroup
\small
\bibliographystyle{alpha}
\bibliography{nonuniform_las_vergnas_counterexample_v24}
\endgroup

\end{document}